\documentclass[11pt]{article}
\usepackage{graphicx}
\usepackage{color}
\usepackage{amsmath}
\usepackage{amssymb}
\usepackage{amscd}
\usepackage{bbm}
\usepackage{tcolorbox}
\usepackage{fancyhdr,a4wide}
\usepackage{amsthm}

\newcommand{\R}{\mathbb{R}}
\newcommand{\inr}[1]{\left\langle #1 \right\rangle}

\newcommand{\E}{\mathbb{E}}

\newcommand{\PP}{\mathbb{P}}

\newcommand{\eps}{\varepsilon}

\newtheorem{Theorem}{Theorem}[section]
\newtheorem{Lemma}[Theorem]{Lemma}

\newtheorem{Proposition}[Theorem]{Proposition}

\theoremstyle{definition}
\newtheorem{Definition}[Theorem]{Definition}
\newtheorem{Remark}[Theorem]{Remark}

\newtheorem{Example}[Theorem]{Example}

\numberwithin{equation}{section}

\def\IND{\mathbbm{1}}

\def\IND{\mathbbm{1}}

\title{Optimal non-gaussian Dvoretzky-Milman embeddings}
\author{
Daniel Bartl\footnote{
University of Vienna, Faculty of Mathematics  (daniel.bartl@univie.ac.at)}
 \and
Shahar Mendelson\footnote{The Australian National University, Centre for Mathematics and its Applications and University of Warwick, Department of Statistics  (shahar.mendelson@anu.edu.au)}
}

\begin{document}
\maketitle

\begin{abstract}
We construct the first non-gaussian ensemble that yields the optimal estimate in the Dvoretzky-Milman Theorem: the ensemble exhibits almost Euclidean sections in arbitrary normed spaces of the same dimension as the gaussian embedding---despite being very far from gaussian (in fact, it happens to be heavy-tailed).
\end{abstract}

\section{Introduction} \label{sec:intro}

The Dvoretzky-Milman Theorem is one of most remarkable results in Asymptotic Geometric Analysis. Dvoretzky proved in \cite{dvoretzky1961results} that $\ell_2$ is \emph{finitely represented} in any infinite dimensional normed space: that any infinite dimensional normed space has subspaces of arbitrarily high (finite) dimension that are almost Euclidean in the \emph{Banach-Mazur sense} (for the definition of that distance see, e.g., \cite{artstein2015asymptotic}). Dvoretzky's proof was finite dimensional: he showed that there is $d=d(n,\eps)$ that tends to infinity with $n$ for any fixed $\eps>0$, such that every $n$-dimensional normed space has a subspace of dimension $d$ that is $1+\eps$ close to Euclidean. 
As it happens, his estimate on $d(n,\eps)$ was suboptimal, and was dramatically improved by Milman in his seminal work \cite{milman1971new}.

Milman showed that any finite dimensional normed space $(\R^n,\|\cdot\|)$ has a \emph{critical dimension} $d^\ast$, and a typical subspace of dimension $c(\eps)d^\ast$---with respect to the Haar measure on the Grassmann manifold of that dimension---is almost Euclidean. Moreover, every normed space is isometric to a space for which $d^\ast \geq c_1 \log n$.

The idea behind Milman's proof is based on the fact that Lipschitz functions are almost constant on the sphere. 
More accurately and in the gaussian version of the theorem, set  $d \leq d^\ast$ and let $G_1,\dots,G_d$ be independent copies of the standard gaussian random vector $G$ in $\R^n$.
One can  show that a typical realization of the function $x \mapsto \|\sum_{i=1}^d x_i G_i\|$ is almost constant on $S^{d-1}$. 
As a result, there is $\Lambda > 0$ such that for every $x \in \R^d$,
\begin{equation} \label{eq:dvor-0-intro}
(1-\eps)\|x\|_2 
\leq \Lambda^{-1} \left\| \sum_{i=1}^d x_i G_i \right\| 
\leq (1+\eps) \|x\|_2.
\end{equation}
Equation \ref{eq:dvor-0-intro} follows from two features of the gaussian vector: firstly, \emph{rotation invariance}, in particular that for any $x \in \R^d$, $\E \|\sum_{i=1}^d x_i G_i\| = \|x\|_2 \cdot \E \|G\|$; and secondly, that each random variable $\|\sum_{i=1}^d x_i G_i\|$ \emph{concentrates sharply} around its mean.

Since Milman's proof---some 50 years ago---the only random ensembles known to exhibit Euclidean subspaces of the optimal (critical) dimension in arbitrary normed spaces were the gaussian ensemble $x \to \sum_{i=1}^d x_i G_i$; and a $d$-dimensional subspace selected according to the Haar measure on the right Grassmann manifold.
Those were the only examples of random ensembles that satisfied both rotation invariance and suitable concentration. 
Our goal in what follows is to construct more general ensembles (in fact, ensembles that can be heavy-tailed) that exhibit Euclidean subspaces of the optimal dimension in arbitrary normed spaces.

\vspace{0.5em}
To formulate Milman's result, consider a normed space $F=(\R^n,\| \cdot \|)$ and denote its unit ball by $K$.
Thus, $K$ is a \emph{convex body}: it is a convex, compact, centrally-symmetric subset of $\R^n$ with a nonempty interior.
Denote by $\inr{\cdot,\cdot}$ the standard Euclidean inner product, and set
\[
K^\circ = \{ y \in \R^n : \inr{x,y} \leq 1 \ {\rm for \ every \ } x \in K \}
\]
to be the polar body of $K$.
Consider the standard gaussian random vector $G$ in $\R^n$, and define the critical dimension of $K$ by
\[
d^\ast(K)=\left(\frac{\E \|G\|}{\sup_{t \in K^\circ} \|t\|_2} \right)^2.
\]
Milman's result (in its gaussian version) is the following.

\begin{Theorem}
\label{thm:DM.classic} 
	For every $\eps>0$ there is a constant $c_1$ depending on $\eps$ and absolute constants $c_2,c_3$ such that the following holds.
	Let  $d\leq c_1(\eps) d^\ast(K)$ and set $G_1,\dots,G_d$ to be independent copies of $G$.
	If $A=\sum_{i=1}^d \inr{e_i,\cdot} G_i\colon\R^d\to\R^n$ then with  probability at least $1-2\exp(-c_2\eps^2 d^\ast(K))$,
	\begin{align*}
	(1-\eps)\E\|G\| \cdot \left( K\cap A\R^d \right)
	\subset AB_2^d
	\subset (1+\eps)\E\|G\| \cdot \left( K\cap A\R^d \right).
	\end{align*}
	Moreover, for any convex body $K$ there is  $T\in GL_n$ such that $d^\ast( TK) \geq c_3\log n$.
\end{Theorem}

\begin{Remark}
\label{rem:critical}
Clearly, for any $T \in GL_n$ the normed spaces whose unit balls are $K$ and $TK$ are isometric. Hence, for the purpose of finding almost Euclidean sections of $K$ we may and do assume without loss of generality that $d^\ast(K) \geq c \log n$.
\end{Remark}	

The proof of Theorem  \ref{thm:DM.classic}  with the optimal $c_1(\varepsilon)\sim\varepsilon^2$ is due to Gordon \cite{gordon1985some} (see also Schechtman's proof \cite{schechtman1989remark}).
A simpler argument that follows Milman's original proof can be found in \cite{pisier1986probabilistic}, leading to the slightly suboptimal estimate of $c_1(\eps) \sim \frac{\eps^2}{\log(1/\eps)}$.
A detailed survey on the Dvoretzky-Milman Theorem can be found in \cite{artstein2015asymptotic}.

The key features of the gaussian ensemble $A$ that lead to the proof of Theorem \ref{thm:DM.classic}, namely, ``gaussian concentration'' and rotation invariance are  rather special.
At least one of the two  fails when it comes to other natural choices of random vectors.
For example, if $A$ is generated by the Bernoulli random vector $(\eps_i)_{i=1}^n$---uniformly distributed in $\{-1,1\}^n$---, and $\|\cdot\|=\|\cdot\|_\infty$, then $x\mapsto \E\|Ax\|$ is not even equivalent to a constant on $S^{d-1}$. The reason for the gap between $\inf_{x \in S^{d-1}} \|Ax\|_\infty$ and $\sup_{x \in S^{d-1}} \|Ax\|_\infty$ is straightforward: if $x=e_i$, $\|Ax\|_{\infty}=\|(\eps_1,...,\eps_n)\|_{\infty}=1$, but if $x$  is in a `diagonal direction', $\|Ax\|_\infty$ is close to $\|G\|_\infty$, and for a typical realization the latter behaves like $\sqrt{\log n}$ (see, e.g., \cite{mendelson2022isomorphic} for a detailed proof).
Also, it is standard to verify that $\|Ax\|$ does not exhibit gaussian concentration around its mean for every fixed $x \in S^{d-1}$ (see \cite{huang2021dimension}).

Given the difficulty in finding ensembles that satisfy the two features, it is natural to ask whether the gaussian ensemble is really the only \emph{optimal Dvoretzky-Milman ensemble}.
That is,  are there other families of random operators that exhibit, \emph{for an arbitrary normed space} and with nontrivial probability, the same behaviour as the gaussian operator: exposing  Euclidean sections of the space's critical dimension.
There are examples of optimal ensembles for restricted classes of normed spaces (e.g.\ spaces with cotype 2, see \cite{mendelson2008subgaussian}), but none of them preforms well in an \emph{arbitrary} normed space.

\vspace{0.5em}
The \emph{key} to our construction of general Dvoretzky-Milman ensembles is factoring through an intermediate space: first use the random operator $\Gamma\colon \R^d\to\R^m$ (with $m=\mathrm{poly}(d)$) that has iid rows $(X_i)_{i=1}^m$ generated by a random vector $X$,
\[
\Gamma=\frac{1}{\sqrt{m}} \sum_{i=1}^m \inr{X_i,\cdot}e_i;
\]
and then apply the random operator
\[ D=\sum_{i=1}^m \inr{e_i,\cdot} Z_i \colon \R^m\to \R^n,\]
where $(Z_i)_{i=1}^m$ are independent copies of a random vector $Z$ in $\R^n$ that has iid coordinates; in particular, $D$ has iid entries.
We will always assume that $(X_i)_{i=1}^m$ and $(Z_i)_{i=1}^m$ are independent.

As a preliminary test case to this idea's validity, let $Z$ be the standard gaussian random vector in $\R^n$.
Clearly $D\Gamma$ has a chance of being an optimal Dvoretzky-Milman ensemble only if $\Gamma$ is an almost isometric Euclidean embedding of $\R^d$ in $\R^m$; otherwise, the embedding would fail even if $\|\cdot\|$ were the Euclidean norm. At the same time, even if $\Gamma$ is an isometry, $D\Gamma$ need not be an optimal Dvoretzky-Milman ensemble for an arbitrary subgaussian\footnote{A random vector $Z$ is called $(L-)$subgaussian if $\|\inr{Z,x}\|_{L_p}\leq  L\sqrt{p} \|\inr{Z,x}\|_{L_2}$ for every $x\in \R^n$ and $p\geq 2$.}, isotropic random vector $Z$: 
as noted previously, if $\Gamma=\mathrm{Id}$, $Z$ is the Bernoulli vector and $\|\cdot\|=\|\cdot\|_\infty$, there is a logarithmic gap between $\inf_{x \in S^{m-1}} \|D\Gamma x\|$ and $\sup_{x \in S^{m-1}} \|D\Gamma x\|$, and the resulting embedding is not even isomorphic with an absolute constant.

\vspace{0.5em}
\begin{tcolorbox}
In a nutshell, we show that these are the only two restrictions.
Essentially, if the matrix $\Gamma$ has iid rows distributed according to a random vector that is  rotation invariant, and is an almost isometric embedding of $\ell_2^d$ in $\ell_2^m$, and $D\colon\R^m \to \R^n$ has iid entries and isotropic subgaussian columns, then with high probability, $(D\Gamma \R^d) \cap K$ is almost Euclidean.
\end{tcolorbox}

Before we list the assumptions that we require and specify the  choices of $d$ and $m$,  let us recall some well-known definitions and facts.

\begin{Definition}
A centred random vector $X$ in $\R^d$ satisfies $L_p-L_2$ norm equivalence with constant constant $L$ if for any $u \in \R^d$,
\begin{align}
\label{eq:def.norm.equiv}
 \|\inr{X,u}\|_{L_p} \leq L \|\inr{X,u}\|_{L_2}.
 \end{align}
Note that if, in addition, $X$ is isotropic (that is, $X$ is centred and its covariance is the identity), then \eqref{eq:def.norm.equiv} is equivalent to $\sup_{u\in S^{d-1}}\|\inr{X,u}\|_{L_p} \leq L$.
\end{Definition}

 Next, given $X_1,...,X_m$ that are independent, selected according to the isotropic random vector $X$, set
\[
\rho_{d,m}=\sup_{u \in S^{d-1}} \left|\frac{1}{m}\sum_{i=1}^m \inr{X_i,u}^2 -1 \right|.
\]
Observe that the singular values of the random matrix $\Gamma=m^{-1/2}\sum_{i=1}^m \inr{X_i,\cdot}e_i$ satisfy that
\[1 -\rho_{d,m} 
\leq \lambda_{\rm{min}}^2(\Gamma) 
\leq \lambda_{\rm{max}}^2(\Gamma) 
\leq 1+\rho_{d,m}.\]
The behaviour of $\rho_{d,m}$ is well-understood in rather general situations (see, e.g., \cite{adamczak2010quantitative,guedon2017interval,mendelson2014singular,tikhomirov2018sample}).
It turns out that under minimal assumptions on $X$ and with high probability, the random matrix $\Gamma$ satisfies the \emph{quantitative Bai-Yin asymptotics} (see \cite{bai1993limit}), namely that $\rho_{d,m} \leq c(d/m)^{1/2}$ for an absolute constant $c$.

\begin{tcolorbox}
The assumptions we require on the random vectors $X$ and $Z$ are as follows:
\begin{enumerate}
\item[(A1)]
 $X$ is an isotropic, rotation invariant random vector in $\R^d$ that satisfies $L_4-L_2$ norm equivalence with constant $L \geq 1$.

Moreover, with probability at least $1-\eta$,  $\rho_{d,m} \leq B(d/m)^{2/\alpha}$ for some constants $B\geq 1$ and $\alpha \geq 4$.
\item[(A2)]
 $Z$ is an isotropic, symmetric, $L$-subgaussian random vector in $\R^n$ that has iid coordinates.
\end{enumerate}
\end{tcolorbox}

A key component in our construction relies on the recent results of \cite{bartl2022structure}.  In a nutshell, the random set $\Gamma S^{d-1}$ has a rather specific structure: the coordinates of each vector $(\inr{X_i,u})_{i=1}^m$ concentrate around a well-determined set of values, endowed by the probability distribution of $X$. More accurately, denote  by $a^\sharp$ the monotone non-decreasing rearrangement of a vector $a\in \R^m$ and set, for $1\leq s\leq m$,
\[
H_{s,m}=\sup_{u\in S^{d-1}} \sup_{|I|\leq s} \left( \frac{1}{m}\sum_{i\in I} \inr{X_i,u}^2 \right)^{1/2}.
\]
Let $F_{\inr{X,u}}^{-1}$ be the (right-)inverse distribution function of $\inr{X,u}$ and put, for $i=1,\dots,m$,
\[
\lambda_i^u=m\int_{(\frac{i-1}{m}, \frac{i}{m}]}  F_{ \inr{X,u} }^{-1}(p) \,dp.
\]

\begin{Proposition}[\cite{bartl2022structure}]
\label{rem:bound.H.via.S.W.trivial} 
	Then there are constants $c_1,c_2,c_3$ that depend on $L, B,\alpha$ such that the following holds.
	Suppose that Assumption (A1) is satisfied with the values $L, B,\alpha$.
	Then, for every $d\geq 1$ and $m\geq c_1 d$,  with probability at least $1-\eta-\exp(- d)$,
\[ \sup_{ u \in S^{d-1}} \left( \frac{1}{m}\sum_{i=1}^m \left| \inr{X_i,u}^\sharp -  \lambda_i^u \right|^2 \right)^{1/2}
\leq  c_2 \left(\frac{d}{m}\right)^{1/\alpha} \log\left(\frac{m}{d}\right), \]
and for every $1\leq s\leq m$,
\[
H_{s,m}
\leq c_3 \left( \left(\frac{s}{m}\right)^{1/4} +  \left(\frac{d}{m}\right)^{1/\alpha} \log\left(\frac{m}{d}\right) \right)  .
\]
\end{Proposition}

Denote by $\Omega(\mathbb{X})$ the event in which both the assertion of  Proposition \ref{rem:bound.H.via.S.W.trivial} holds and $\rho_{d,m}\leq  B(d/m)^{2/\alpha}$.

\vskip0.4cm
Note that if $X$ is rotation invariant then $\lambda^u=\lambda^v$ for every $u,v\in S^{d-1}$. Thus, Proposition \ref{rem:bound.H.via.S.W.trivial} implies that after replacing the vectors in $\Gamma S^{d-1}$  by their monotone non-decreasing ordering, the resulting set has a small Euclidean diameter.
This is the key feature that `generates' sufficient invariance.
As a result there is no need to assume that the random vector $Z$ is rotation invariant.

\vskip0.4cm
We are now ready to formulate our main result.
Recall that $\Gamma\colon \R^d\to\R^m$ is the random operator whose rows are $(X_i)_{i=1}^m$ and  $D\colon  \R^m\to \R^n$ is the random operator whose  columns are $(Z_i)_{i=1}^m$.
And,  assume (as we may---see Remark \ref{rem:critical}) that  $d^\ast(K)\geq C_0\log n$ for an absolute constant $C_0$.

\begin{tcolorbox}
\begin{Theorem}
\label{thm:DM-main}
There are constants $c_1\leq \frac{1}{2}$, $c_2,\dots,c_8$ that depend on $L,B,\alpha,C_0$ such that the following holds.
Suppose that Assumptions (A1) and (A2) are satisfied with the values  $L,B,\alpha$.
Let $\varepsilon\leq c_1$,  $n\geq c_2$,   
\[ m=c_3  (d^\ast(K))^{\alpha}
\quad\text{and}\quad
d\leq c_4 \frac{\varepsilon^2}{\log(1/\varepsilon)} d^\ast(K).\]
 
Then there is some $\Lambda \in [c_5 \E\|G\|, c_6\E\|G\|]$ such that for any realization $(X_i)_{i=1}^m \in \Omega(\mathbb{X})$, with $\PP_Z$-probability at least $1-2\exp(-c_7\eps^2 d^\ast(K))$,
\[
\sup_{u \in S^{d-1}} \left|\frac{\|D\Gamma u\|}{\Lambda}-1 \right| \leq c_8\eps.
\]
\end{Theorem}
\end{tcolorbox}


\begin{Remark}
The assumption that $X$ satisfies $L_4-L_2$ norm equivalence is necessary: by \cite{bai1988note}, if $X$ has iid coordinates distributed according to a mean zero, variance $1$ random variable that is not bounded in $L_4$, then $\rho_{d,m} \to \infty$ as $d$ and $m$ tend to infinity while keeping their ratio constant.
Hence, without the norm equivalence assumption $\rho_{d,m}$ is not well-behaved and $\Gamma$ cannot even be an isomorphic embedding of $\ell_2^d$ in $\ell_2^m$.
At the same time, if $\Gamma$ satisfies the quantitative Bai-Yin asymptotics (and as noted previously, that is the case under minimal assumptions on $X$) then (A1) holds for $\alpha=4$.
\end{Remark}

\begin{Remark}
	The idea of factoring through a higher dimensional space was used in \cite{mendelson2022isomorphic} to prove the existence of \emph{isomorphic} non-gaussian Dvoretzky-Milman ensembles.
	The argument was based on establishing lower bounds on the expected suprema of certain Bernoulli processes.
	Such bounds are possible only up to a multiplicative constant, and thus cannot be used to prove that an ensemble yields an almost isometric embedding --- regardless of the choice of $X$.
\end{Remark}

\vspace{0.5em}
The rest of the article is devoted to the proof of Theorem \ref{thm:DM-main}.
Before we dive into technical details, let us give an example of such an ensemble, for a heavy-tailed choice of $X$ and   $Z=(\eps_1,...,\eps_n)$ --- the Bernoulli vector.

\begin{Example} \label{ex:DM.concrete}
Let $K$ be the unit ball of $(\R^n,\| \cdot \|)$  and set $d\leq c_0\frac{\eps^2 }{\log(1/\eps)}d^\ast(K)$.
Let $W$ be distributed uniformly in $S^{d-1}$ and consider a symmetric random variable $v$ for which $\E v^2=1$ and $\|v\|_{L_7} \leq L$ (but potentially $v\notin L_q$ for  $q>7$).

Setting $X=\sqrt{d} Wv$, it is straightforward to verify that $X$ is a centred, isotropic random vector that is rotation invariant and satisfies $L_7-L_2$ norm equivalence with  a constant $L'\sim L$.
By \cite{tikhomirov2018sample}, with probability at least $1-\frac{1}{d}$,
\[ \rho_{d,m} \leq c_2(L) \left( \max_{1\leq i\leq m} \frac{ \|X_i\|_2^2 }{m} + \sqrt \frac{d}{m} \right). \]
	Moreover, for  $m \geq  d^4$, by Markov's inequality,
	\begin{align}
	\label{eq:maximum.bound.dm.1.4}
	 \PP\left( \max_{1\leq i\leq m} \|X_i\|_2^2 \leq  L^2 \sqrt{dm} \right)
	\geq 1- m\PP\left(  |v| \geq L \left(\frac{m}{d}\right)^{1/4} \right)
	\geq 1-m\frac{d^{q/4}}{m^{q/4}}
	\geq 1- \frac{1}{d}
	\end{align}
	where we used that $q\geq 7$ in the last inequality.
	Hence, with probability at least $1-\frac{2}{d}$,  $\rho_{d,m}\leq  c_3(L) \sqrt{d/m}$, and Assumption (A1) is satisfied with $\eta=\frac{2}{d}$,  $\alpha=4$ and a constant $B$ depending only on $L$.	

Now put $m \sim (d^*(K))^4$ and consider a realization of $\Gamma$ for which the assertion of Proposition \ref{rem:bound.H.via.S.W.trivial} holds and $\rho_{d,m}\leq c_3\sqrt{d/m}$, i.e.\ $(X_i)_{i=1}^m\in\Omega(\mathbb{X})$.
By Theorem \ref{thm:DM-main}, for every such realization of $\Gamma$ and with high $\PP_Z$-probability,  the section $D\Gamma \R^d \cap K$ is $(1\pm c\eps)$-Euclidean.
In particular, $D\Gamma$ is an optimal Dvoretzky-Milman ensemble despite being `very far' from gaussian.
\end{Example}

\begin{Remark}
	Let us note that the set of Euclidean subspaces generated by our operator $D\Gamma$ significantly differs from those generated by the standard  approach.
	Have, for example,  the setting in  Example \ref{ex:DM.concrete}.
	We claim that the Haar measure of all $d$-dimensional subspaces generated by $D\Gamma$ is zero.
	
	Indeed, each subspace we generate is $d$-dimensional and contained in the span of $m$ vertices of the combinatorial cube in $\R^n$. 
	Since there are $\binom{2^n}{m}$ such $m$-dimensional subspaces,  it is enough to show that the probability (with respect to the Haar measure on the Grassmann manifold $G_{d,m}$) that a $d$-dimensional subspace of $\R^n$ is contained in a \emph{fixed} $m$-dimensional subspace of $\R^n$ is 0. 
	By rotation invariance, we may assume that the $m$-dimensional subspace is ${\rm span}(e_1,...e_m)$.
	Moreover, 
	\[ \PP_{ G_{d,m} } \left( E \in {\rm span}(e_1,...e_m) \right)
	=\PP_{O_n}\left( Oe_i \in {\rm span}(e_1,...e_m) \text{ for }1\leq i\leq d \right)\]
	where $\PP_{O_n}$ is the Haar measure on the orthonormal group $O_n$.
	But $Oe_i$ is distributed as $G/ \|G\|_2$, and almost surely all coordinate of the standard gaussian are nonzero.
\end{Remark}

\vspace{0.5em}
Recall that the operator $D$ in Example \ref{ex:DM.concrete}, generated by the Bernoulli vector, does not `act well' on the entire sphere $S^{m-1}$: if $\| \cdot \|=\| \cdot \|_\infty$, there is a logarithmic gap between $\inf_{x \in S^{m-1}} \|Dx\|$ and $\sup_{x \in S^{m-1}} \|Dx\|$.
However, $D$ acts well on the $d$-dimensional section of $S^{m-1}$ endowed by $\Gamma S^{d-1}$.
Thus, the meaning of $\Gamma S^{d-1}$ being in a \emph{good position} is that, regardless of the choice of $Z$ or of the normed space $(\R^n, \| \cdot \|)$, with high probability the oscillation of $x \to \|Dx\|$ on $\Gamma S^{d-1}$ is minimal.

The proof that a typical realization of $\Gamma S^{d-1}$ is indeed in a good position relies heavily on Proposition \ref{rem:bound.H.via.S.W.trivial}.
Obviously, the random operator $D\Gamma$ is far from being gaussian, but we shall show that thanks to Proposition \ref{rem:bound.H.via.S.W.trivial}, $D\Gamma$ still exhibits some key gaussian features.
Specifically, because vectors in $\Gamma S^{d-1}$ inherit $X$'s rotation invariance in the sense described in Proposition \ref{rem:bound.H.via.S.W.trivial}, the function $ x\mapsto \E_Z\| Dx\|$ is almost constant on $\Gamma S^{d-1}$.
And, in addition, Proposition \ref{rem:bound.H.via.S.W.trivial} is the crucial ingredient in establishing that for each $x\in \Gamma S^{d-1}$ the random variable $\|Dx\|$ exhibits `gaussian' concentration  around its mean (a behaviour that does not hold for an arbitrary $y \in S^{m-1}$).

\section{Preliminaries and a reduction step}

We start with a word about \emph{notation}. $\|\cdot\|_2$ is the  Euclidean norm  and  $\inr{\cdot,\cdot}$  is the standard inner product---though  in what follows we do not specify the (finite) dimension of the underlying space.
Throughout, $c,c_0,c_1,C,C_0,C_1,\dots$ are absolute  constants whose values may change from line to line.
If a constant $c$ depends on a parameter $a$,  we write $c=c(a)$; and if $cA \leq B \leq CA$ for absolute  constants $c$ and $C$, that is denoted by $A \sim B$.
The cardinality of finite sets is denoted by $|\cdot|$.
For two  independent random vectors $X$ and $Y$, $\E_X$ is the expectation with respect to $X$;  i.e., if $X$ is  distributed according to $\mu$ then $\E_X f(X,Y)=\int f(x,Y)\,\mu(dx)$.

\subsection{Preliminaries}
\label{sec:preliminary.obs}

Let us consider some implications of the assumptions made in Theorem \ref{thm:DM-main}.
We can and do assume that $\eps^2 d^\ast(K)\geq 1$---otherwise, the claim is trivially true.

Let $\varepsilon\leq c_0$ for a constant $c_0=c_0(L)\leq 1/2$.
The choice of $c_0$ is specified in Lemma \ref{lem:X.SB}.
Let $c_1$ and $c_2$ be well-chosen constants that depend only on $L$, and set
\begin{equation} \label{eq:s}
s=c_1\frac{d^\ast(K)}{\log(em/d^\ast(K))} ,
\qquad
d \leq c_2 \frac{\eps^2}{\log(1/\eps)}d^\ast(K).
\end{equation}
The choice of $c_1$ and $c_2$ is specified in Theorem \ref{thm:conc}.

Recall that $\Omega(\mathbb{X})$ is the event in which the assertion of Proposition \ref{rem:bound.H.via.S.W.trivial} holds and $\rho_{d,m}\leq B(d/m)^{2/\alpha}$.
Consider  $(X_i)_{i=1}^m\in\Omega(\mathbb{X})$.
Since $X$ is rotation invariant it follows that for every $u,v\in S^{d-1}$, $\lambda^u=\lambda^v=\lambda$; hence
\begin{align*}
\sup_{ u \in S^{d-1}} \left( \frac{1}{m}\sum_{i=1}^m \left| \inr{X_i,u}^\sharp - \lambda_i  \right|^2 \right)^{1/2}
&\leq  c(L,B) \left(\frac{d}{m}\right)^{1/\alpha} \log\left(\frac{m}{d}\right)  
\end{align*} 
and
\begin{align*}
H_{s,m}
&\leq c'(L,B)\left( \left(\frac{s}{m}\right)^{1/4} +  \left(\frac{d}{m}\right)^{1/\alpha} \log\left(\frac{m}{d}\right) \right).
\end{align*}
Thus, for any given constant  $\beta>0$, if
\begin{equation} \label{eq:cond-on-m}
m = c_3(c_1,c_2,L,B,\alpha,\beta) \cdot (d^\ast(K))^{\alpha},
\end{equation}
then
\begin{equation} 
\label{eq:SW-2-est-DM}
\frac{H_{s,m}}{\sqrt{s}}
\leq \frac{\beta}{d^\ast(K)} ,\qquad
\sup_{ u\in S^{d-1}} \left( \frac{1}{m}\sum_{i=1}^m \left| \inr{X_i,u}^\sharp - \lambda_i  \right|^2 \right)^{1/2}
\leq \frac{1}{\sqrt{ d^\ast}}
\leq\eps,
\end{equation}
and 
\begin{equation} \label{eq:SW-2-est-DM.2}
\sup_{ u\in S^{d-1}} \left( \frac{1}{m}\sum_{i=1}^m \left| \inr{X_i,u}^\sharp - \inr{X_i,v}^\sharp  \right|^2 \right)^{1/2}
\leq 2\varepsilon.
\end{equation}
The choice of  $\beta$ is specified in  Theorem \ref{thm:conc}.
Moreover, since $\rho_{d,m} \leq B(d/m)^{2/\alpha}$,  we may also assume that $\rho_{d,m}\leq 1$.
Thus,
\begin{align}
\label{eq:SW-2-est-norm}
\sup_{u\in S^{d-1}} \frac{1}{m} \sum_{i=1}^m \inr{X_i,u}^2 \leq 2.
\end{align}

Finally,   $d^\ast(K)\geq C_0 \log n$ for an absolute constant $C_0$.
We only consider $n$ sufficiently large, specifically,  $n\geq c_4(L,\alpha)$ for a suitable constant $c_4$ that is specified in Theorem \ref{thm:lower.bound.expectation}.

\subsection{A reduction step}

Let $W \subset S^{d-1}$ be  a maximal $\eps$-separated subset with respect to the Euclidean norm.
Recall that $Z$ is an isotropic, symmetric, $L$-subgaussian random vector taking  values in $\R^n$ and has iid entries\footnote{The fact that $Z$ has iid coordinates is used only in Section \ref{sec:exp.lower} where we prove condition \eqref{eq:conc-2} below.}, let  $(Z_i)_{i=1}^m$ be independent copies of $Z$ and define  $\Psi\colon\R^d\to\R$ by
\[
\Psi(u)= \left\|\frac{1}{\sqrt m} \sum_{i=1}^m \inr{X_i,u}Z_i \right\|.
\]

\begin{Lemma}
\label{lemma:reduction}
	Let $\eps\leq 1/2$ and assume that there are $v \in S^{d-1}$, $\eta\in[0,1]$, constants $c_1,c_2,c_3$, and an event ${\cal A}$ such that for $(X_i)_{i=1}^m\in \mathcal{A}$ the following holds.
\begin{enumerate}
\item[(1)]
\underline{Uniform concentration on the net:} with $\PP_Z$-probability at least $1-\eta$,
\begin{equation} \label{eq:conc-1}
\max_{w \in W} \left|  \Psi(w) - \E_Z\Psi(w) \right|
\leq c_1 \eps \E\|G\|.
\end{equation}
\item[(2)]
\underline{A lower bound on the conditional expectation:}
setting $\Lambda =\E_Z \Psi(v)$,
\begin{equation} \label{eq:conc-2}
\Lambda \geq c_2 \E \|G\|.
\end{equation}
\item[(3)]
\underline{Almost constant conditional expectation on the net:} 
\begin{equation} \label{eq:conc-3}
\max_{w\in W} \left| \E_Z \Psi(w)- \Lambda \right|
\leq c_3 \eps \E \|G\|.
\end{equation}
\end{enumerate}
Then, there is a constant  $c=c(c_1,c_2,c_3)$ such that for $(X_i)_{i=1}^m \in\mathcal{A}$, with $\PP_Z$-probability at least $1-\eta$,
\[
\sup_{u \in S^{d-1}} \left| \Psi(u) - \Lambda \right| \leq c \eps \Lambda.
\]
\end{Lemma}

We show that (1)-(3) are satisfied with $\mathcal{A}=\Omega(\mathbb{X})$.
Specifically, (3) is established in Theorem \ref{thm:expectation.almost.constant}; (1) is verified  in Theorem \ref{thm:conc}; and (2) is proved in Theorem \ref{thm:lower.bound.expectation}.

\begin{proof}[Proof of Lemma \ref{lemma:reduction}]
	Since $W$ is an $\eps$-net in $S^{d-1}$, by a standard successive approximation argument, for every $u\in S^{d-1}$ there is a sequence $(w_j^u)_{j\geq 0}$ in $W$ such that
	\begin{align}
	\label{eq:u.approx.net}
	 u = \sum_{j=0}^\infty \gamma_j w_j^u,
	 \end{align}
	where $\gamma_0=1$ and $|\gamma_j| \leq \eps^j$ for $j \geq 1$.	
	Using the linearity of $\frac{1}{\sqrt m} \sum_{i=1}^m \inr{\cdot,X_i} Z_i$  it follows that
	\begin{align*}
	|\Psi(u)-\Psi(w_0^u)|
	\leq  2 \varepsilon \sup_{w\in W} \Psi(w).
	\end{align*}
	Let us show that $\sup_{w\in W} \Psi(w)\leq c\Lambda$ for a constant $c=c(c_1,c_2,c_3)$, and in particular
	\begin{align}
	\label{eq:Psi.approx.net}
	\sup_{u\in S^{d-1}}|\Psi(u)-\Psi(w_0^u)|
	\leq  2 c \varepsilon  \Lambda.
	\end{align}
	To that end, for $w\in W$
	\[\Psi(w)
	\leq |\Psi(w)-\E_Z\Psi(w)| + |\E_Z\Psi(w)- \Lambda| + \Lambda
	=(\ast)+ (\ast\ast) +\Lambda.\]
	By combining \eqref{eq:conc-1} and \eqref{eq:conc-2}, $(\ast)\leq \frac{c_1}{c_2}\eps\Lambda$;  and by invoking \eqref{eq:conc-2} and \eqref{eq:conc-3}, $(\ast\ast)\leq \frac{c_3}{c_2}\eps\Lambda$.
	
	Finally, set $u\in S^{d-1}$  and consider  $(w_j^u)_{j\geq 0}$ as in \eqref{eq:u.approx.net}.
	Using the estimates on $(\ast)$ and $(\ast\ast)$ again,
	\begin{align*}
	| \Psi(u)- \Lambda |
	&\leq 	| \Psi(u)-\Psi(w_0^u)|+	| \Psi(w_0^u) -  \E_Z \Psi(w_0^u)|+ 	|\E_Z\Psi(w_0^u)-\Lambda|\\
	&\leq c'(c_1,c_2,c_3)\eps\Lambda.
	\qedhere
	\end{align*}
\end{proof}


\section{The conditional expectation is almost constant---proof of \eqref{eq:conc-3}}

Let us show that for every realization $(X_i)_{i=1}^m$ that satisfies \eqref{eq:SW-2-est-DM.2}, the oscillation of $u\mapsto \E_Z \Psi(u)$ is at most of order $\varepsilon \E \|G\|$. 
To that end, define the distance $w_2$ on $\R^m$ by
\begin{align}
\label{eq:def.w2}
w_2(x,y)= \min_\sigma \left( \sum_{i=1}^m \left( x_i - y_{\sigma(i)} \right)^{2} \right)^{1/2} 
,
\end{align}
where the minimum is taken over all permutations $\sigma$ of $\{1,\dots,m\}$.

\begin{Remark}
The reason for the name $w_2$ is the obvious connection \eqref{eq:def.w2} has with  $\mathcal{W}_2$---the second order Wasserstein distance (see Remark \ref{rem:Wasserstein}), namely
\[
w_2(x,y)=\mathcal{W}_2\left(\frac{1}{m}\sum_{i=1}^m\delta_{x_i}, \frac{1}{m}\sum_{i=1}^m\delta_{y_i}\right).
\]
\end{Remark}

\begin{Lemma}
\label{lem:permutation.controls}
	There is an absolute constant $c$ such that for every $x,y\in\mathbb{R}^m$,
	\[ \left| \E \left\| \sum_{i=1}^m x_i Z_i \right\| - \E \left\| \sum_{i=1}^m  y_iZ_i \right\| \right|
	\leq c L \E \|G\|  \cdot w_2(x,y).  \]
\end{Lemma}
\begin{proof}
	Fix $x,y\in\R^m$ and let $\sigma$ be a permutation of $\{1,\dots,m\}$.
	Note that $\sum_{i=1}^m  y_iZ_i $ has the same distribution as $\sum_{i=1}^m  y_{\sigma(i)} Z_i$, and
	by gaussian domination (which is an immediate consequence of Talagrand's majorizing measure theorem, see, e.g., \cite[Theorem 2.10.11]{talagrand}), if $G_1,\dots,G_m$ are independent copies of $G$ then for any $(w_i)_{i=1}^m\in \R^m$,
	\[ \E \left\| \sum_{i=1}^m  w_i Z_i \right\|
	\leq cL \E \left\| \sum_{i=1}^m  w_i G_i \right\|
	= cL \|w\|_2  \E \left\| G\right\|. \]
	Therefore,
\begin{align*}
	\E \left\| \sum_{i=1}^m  x_iZ_i \right\|
	&\leq  \E \left\| \sum_{i=1}^m  (x_i - y_{\sigma(i)})Z_i \right\| + \E \left\| \sum_{i=1}^m  y_{\sigma(i)} Z_i \right\| \\
	&\leq cL \left( \sum_{i=1}^m \left( x_i - y_{\sigma(i)} \right)^{2}\right)^{1/2} \E\|G\| +  \E \left\| \sum_{i=1}^m  y_iZ_i \right\|,
\end{align*}
	and the claim  follows.
\end{proof}

In view of Lemma \ref{lem:permutation.controls}, to ensure that $u\mapsto \E_Z \|D\Gamma u\|$ is almost constant on $S^{d-1}$ it suffices to show that the set
\[ \left(\Gamma S^{d-1}\right) = \left\{ \left( \frac{ \inr{X_i,u} }{ \sqrt m} \right)_{i=1}^m :  u\in S^{d-1}\right\} \]
has a small diameter with respect to the distance $w_2$.
That follows immediately from \eqref{eq:SW-2-est-DM.2}: indeed, for every $x,y\in \Gamma S^{d}$, $w_2(x,y)\leq 2\varepsilon$.

In particular, we have the following.

\begin{Theorem}
\label{thm:expectation.almost.constant}
	There is a constant $c$ that depends on $L$  such that for every $(X_i)_{i=1}^m\in \Omega(\mathbb X)$,
	\[  \sup_{u,v\in S^{d-1}}  \left| \E_Z\Psi(u) - \E_Z\Psi(v) \right|
	\leq c \eps \E\|G\|.
\]
\end{Theorem}

\begin{Remark}
\label{rem:Wasserstein}
	The condition in Theorem \ref{thm:DM-main} that $X$ is rotation invariant can be relaxed in the following sense:
	For $u\in S^{d-1}$, let $\mu^u$ be the law of $\inr{X,u}$.
	The (second order) Wasserstein distance between two probability measures $\alpha,\beta$ on $\R$ with finite second moment is
	\[ \mathcal{W}_2(\alpha,\beta)
	= \inf_\pi \left(\int_{\R\times\R} (x-y)^2\,\pi(dx,dy) \right)^{1/2},
	\]
	where the infimum is taken over all probability measures $\pi$ on $\R\times \R$ with first marginal $\alpha$ and second marginal $\beta$.
	Using this notation, the assertion of Theorem \ref{thm:DM-main} remains valid under the assumption that  $\sup_{u,v\in S^{d-1}} \mathcal{W}_2(\mu^u,\mu^v)\leq \varepsilon$ (clearly, if $X$ is rotation invariant then for any $u,v\in S^{d-1}$, $\mathcal{W}_2(\mu^u,\mu^v)=0$).
	However, as optimising the choice of $X$ in this sense is  not the focus of the article, we shall not pursue this point further.
\end{Remark}

\section{Uniform concentration on the net---proof of  \eqref{eq:conc-1}}

The second component needed in the proof of Theorem \ref{thm:DM-main} is concentration on the net.
Recall that $W\subset S^{d-1}$ is a maximal $\eps$-separated subset with respect to the Euclidean norm, and we need to show that there is a number $\Lambda$ (that happens to be proportional to $\E\|G\|$), such that  for every realization $(X_i)_{i=1}^m\in\Omega(\mathbb{X})$, with high probability with respect $(Z_i)_{i=1}^m$,
\begin{equation} \label{eq:goal}
\sup_{u \in W} \left| \Psi(u) - \Lambda \right| 
\leq c(L) \eps \Lambda.
\end{equation}


\begin{Theorem} \label{thm:conc}
	There are constants $c_1,\dots, c_6,\beta$ depending only on $L$ such that the following holds.
	Let $\eps\leq 1/2$ and assume that 
	\begin{align}
	\label{eq:ass.on.m.in.DM.thm}
d\leq c_1 \frac{\eps^2}{\log(1/\eps)} d^\ast(K) \ \ \ {\rm and} \ \ \ m\in [c_2 d^\ast(K), c_3 \exp(d^\ast(K))].
	\end{align}
	Consider $s=c_4\frac{ d^\ast(K)}{\log(em/d^\ast(K))}$ and with that choice of $s$, fix a realization $(X_i)_{i=1}^m$ satisfying that
	\begin{align}
	\label{eq:ass.on.X.in.DM.thm}
	\begin{split}
	\sup_{u\in B_2^d}\frac{1}{m}\sum_{i=1}^m\inr{X_i,u}^2
	&\leq 2 	\quad\text{and}\\
	H_{ s ,m}
	=\sup_{u \in B_2^d} \max_{ |I|=s} \left(\frac{1}{m} \sum_{i \in I} \inr{X_i,u}^2 \right)^{1/2}
	&\leq \beta  \frac{ \sqrt s}{  d^\ast(K)} .
	\end{split}
	\end{align}
	Then with $\PP_Z$-probability at least $1-2\exp(-c_5\varepsilon^2 d^\ast(K))$,
	\[\max_{u \in W} \left|  \Psi(u) - \E_Z\Psi(u) \right|
	\leq c_6 \eps \E\|G\|.\]
\end{Theorem}

\vspace{0.5em}

\noindent
\underline{On the assumptions in Theorem \ref{thm:conc}:}
Observe that all of its assumptions are satisfied by the choices made in Theorem \ref{thm:DM-main} and for $(X_i)_{i=1}^m \in \Omega(\mathbb{X})$. 
Indeed,  
\begin{enumerate}
\item 
\eqref{eq:ass.on.X.in.DM.thm} follows from \eqref{eq:SW-2-est-DM} and \eqref{eq:SW-2-est-norm}.
\item 
The condition on $m$ holds because $d^\ast(K)\geq C(L)$ for a large constant $C$ that we are free to choose. And since $m$ is proportional to $(d^\ast(K))^\alpha$  for $\alpha\geq 4$ (see \eqref{eq:cond-on-m}), it follows that  \eqref{eq:ass.on.m.in.DM.thm} is satisfied as well.
\end{enumerate}

\vspace{0.5em}
The proof of Theorem \ref{thm:conc} is somewhat involved.
We first present its `road map' and only then turn to the technical details.

\subsection{Highlights of the proof of Theorem \ref{thm:conc}}

	It is worth stressing once again that the proof of Theorem \ref{thm:conc} does not simply follow from gaussian concentration even though $Z$ is subgaussian.
	In fact, it was recently shown that gaussian-like concentration in a subgaussian setup is false in general, see \cite{huang2021dimension}.	
	As we explain in what follows, what saves the day is the particular structure of the  vectors $(\inr{X_i,u})_{i=1}^m$.

The idea is to write for $u \in W$,
\[
\frac{1}{\sqrt{m}}\sum_{i=1}^m \inr{X_i,u}Z_i
=\clubsuit_u+\diamondsuit_u+\heartsuit_u,
\]
where $\|\clubsuit_u\|$ and $\|\diamondsuit_u\|$  are ``small", while $\|\heartsuit_u\|$ concentrates sharply around its mean.
The vectors $\clubsuit_u$ are obtained by truncating the vectors $(\inr{X_i,u})_{i=1}^m$, and $\diamondsuit_u$ is obtained from a truncation of $(\|Z\|_i)_{i=1}^m$.

The wanted decomposition is achieved in three steps. Recall that $K$ is the unit ball of the norm $\|\cdot\|$ and $K^\circ$ its polar body; in particular $\|\cdot\|=\sup_{t\in K^\circ} \langle \cdot, t\rangle$, and for every $I\subset\{1,\dots,m\}$,
\[ \left\| \frac{1}{\sqrt{m}}\sum_{i\in I} \inr{X_i,u}Z_i \right\|
=\sup_{t\in K^\circ} \frac{1}{\sqrt{m}}\sum_{i\in I} \inr{X_i,u}\inr{Z_i,t}. \]

\vspace{0.8em}
\noindent
\underline{\bf Step 1 --- $\clubsuit_u$:} \emph{Removing the large coordinates of $(\inr{X_i,u})_{i=1}^m$.}

Let $I_{u,s} \subset \{1,\dots,m\}$ be the set of indices corresponding to the $s$ largest coordinates of $(|\inr{X_i,u}|)_{i=1}^m$,
and put
\[
\clubsuit_u = \frac{1}{\sqrt{m}} \sum_{i \in I_{u,s}} \inr{X_i,u} Z_i.
\]

\begin{tcolorbox}
	By the assumption on $(X_i)_{i=1}^m$, the term $\|( \frac{1}{\sqrt m} \inr{X_i,u})_{i \in I_{u,s}}\|_2\leq H_{s,m}$ is small.
	We will show that as a result, $\|\clubsuit_u\|$ must be  small as well.
\end{tcolorbox}

\vspace{0.5em}
Following Step 1, one is  left with $\frac{1}{\sqrt{m}} \sum_{i \in I_{u,s}^c} \inr{X_i,u} Z_i$ to deal with. Set
\[
\xi_u= \frac{1}{\sqrt{m}} \max_{i \in I_{u,s}^c} |\inr{X_i,u}|,\]
and note that
\[\xi_u\leq \frac{H_{s,m}}{\sqrt{s}}  . \]
Thus, a bound on $H_{s,m}$ leads to a uniform bound on $\xi_u$, a fact that will prove to be useful in Step 2 and Step 3.

\vspace{0.8em}
\noindent
\underline{\bf Step 2 --- $\diamondsuit_u$:} \emph{Truncation of $\|Z_i\|$.}

Fix $r$ to be named in what follows.
For a well-chosen function $\phi(r)$, denote by $J_{u,r} \subset \{1,\dots,m\}$  the (random!) set of indices $i$ such that $\| \frac{1}{\sqrt m} \inr{X_i,u} Z_i\| \geq \xi_u \phi(r)$.
We will show that with high probability with respect to $Z$, the sets $I_{u,s}^c\cap J_{u,r}$ consist of at most $r$ elements.

Put
\[
\diamondsuit_u = \frac{1}{\sqrt{m}} \sum_{ i \in I_{u,s}^c \cap J_{u,r}} \inr{X_i,u}Z_i.
\]
The choice of $\diamondsuit_u$ is useful in two important aspects:
\begin{tcolorbox}
\begin{description}
\item{$(1)$} {\it Control in $\ell_\infty$:} For every  $i \in I_{u,s}^c \cap J_{u,r}$,  we have $\frac{1}{\sqrt{m}}|\inr{X_i,u}| \leq \xi_u$.
\item{$(2)$} {\it ``Short support":} since $|I_{u,s}^c \cap J_{u,r}|  \leq r$,  $\diamondsuit_u$ is the sum of a few terms of the form $\frac{1}{\sqrt{m}}\inr{X_i,u}Z_i$.
\end{description}
\end{tcolorbox}
It turns out that $(1)$ and $(2)$ are enough to ensure that $\|\diamondsuit_u\|$ is sufficiently small.

\vspace{0.8em}
\noindent
\underline{\bf Step 3 --- $\heartsuit_u$:} \emph{Concentration.}

Finally, one has to show that the norm of
\[
\heartsuit_u = \frac{1}{\sqrt{m}} \sum_{ i \in I_{u,s}^c \cap J_{u,r}^c} \inr{X_i,u}Z_i
\]
concentrates sharply around its mean.
Note that here there is sufficient control on
\[
\max_{i \in I_{u,s}^c \cap J_{u,r}^c} \|\inr{X_i,u} Z_i\| :
\]
thanks to the choice of $I_{u,s}$ and $J_{u,r}$, both $|\inr{X_i,u}|$ and $\|Z_i\|$ are well bounded for $i\in I_{u,s}^c \cap J_{u,r}^c$.

\begin{tcolorbox}
That will be enough to show that $\|\heartsuit_u\|$ (as a function of $(Z_i)_{i=1}^m$) concentrates around $\E_Z\|\heartsuit_u\|$. Moreover, thanks to the estimates in Step 1 and Step 2, $\E_Z \|\heartsuit_u\|$ is close to $\Lambda=\E_Z\Psi(u)$.
\end{tcolorbox}

\subsection{Preliminary structural estimates for Theorem \ref{thm:conc}}

The following lemma contains well-known facts on subgaussian processes.

\begin{Lemma}
\label{lem:sub-Gaussian.sparse}
	There are absolute constants $C_0,C_0',C_1,C_1'$ such that for every $1\leq r\leq m$ the following hold.
	Let $Z$ be an isotropic, $L$-subgaussian random variable in $\R^n$, let $Z_1,\dots,Z_m$ be independent copies of $Z$ and set ${\cal R}(K^\circ) = \sup_{t\in K^\circ} \|t\|_2$.
	Then:
\begin{enumerate}
\item[(1)]
 With probability at least $1-2\exp(-C_0' r\log(em/r))$,
\begin{equation} \label{eq:conc-event-2}
\sup_{t \in K^\circ} \max_{|J|=r} \left(\sum_{j \in J} \inr{Z_j,t}^2 \right)^{1/2} \leq C_0 L \left(\E \|G\| + {\cal R}(K^\circ) \sqrt{r \log(em/r)} \right).
\end{equation}
\item[(2)]
 We have that
\begin{equation} \label{eq:conc-mean-2}
\left( \E \sup_{t \in K^\circ} \max_{|J|=r} \sum_{j \in J} \inr{Z_j,t}^2 \right)^{1/2} \leq C_0 L\left(\E \|G\| + {\cal R}(K^\circ) \sqrt{r \log(em/r)} \right).
\end{equation}
\item[(3)]
With probability at least $1-2\exp(-C_1'r\log (em/r))$,
\begin{equation} \label{eq:conc-event-norm-2}
\left| \left\{ i\in\{1,\dots,m\} : \|Z_i\| \geq C_1 L \left(\E\|G\| + {\cal R}(K^\circ) \sqrt{\log\left( em/r\right)  } \right) \right\} \right| \leq r.
\end{equation}
\end{enumerate}
\end{Lemma}

The proof of Lemma \ref{lem:sub-Gaussian.sparse} is standard, and it is outlined for the sake of completeness.

\begin{proof}[Sketch of the proof of Lemma \ref{lem:sub-Gaussian.sparse}]
Let $\Sigma_r\subset S^{m-1}$  be the set of $r$-sparse unit vectors, that is,  $\Sigma_r=\{x\in S^{m-1} : |\{ i : x_i \neq 0 \}|\leq r\}$.
Set $\Sigma_r'\subset\Sigma_r$ to be a minimal $1/10$-cover of $\Sigma_r$ with respect to the Euclidean norm.
Because $\Sigma_r$ is the union of $\binom{m}{r}$ Euclidean spheres, a standard volumetric estimate shows that $|\Sigma_r'|\leq \exp(c_0r\log\frac{em}{r})$ for an absolute constant $c_0$ (see, e.g., \cite[Corollary 4.1.15]{artstein2015asymptotic}).
A convexity argument (just as in the proof of Lemma \ref{lemma:reduction}) implies that
\begin{align*}
\sup_{t \in K^\circ} \max_{|J|=r} \left(\sum_{j \in J} \inr{Z_j,t}^2 \right)^{1/2}
&=\sup_{x\in \Sigma_r} \left\| \sum_{j=1}^m x_j Z_j\right\|
\leq 2\max_{x\in \Sigma_r'} \left\| \sum_{j=1}^m x_j Z_j\right\|.
\end{align*}
It is straightforward to verify that each random vector $ \sum_{j=1}^m x_j Z_j$ is $c L\|x\|_2 $-subgaussian and isotropic; thus, by gaussian domination, there is an absolute constant $c_1$ such that for every $x\in \Sigma_r'$ and $p \geq 1$, 
\begin{align}
\left(\E\left \| \sum_{j=1}^m x_j Z_j \right\|^p\right)^{1/p} &\leq c_1 L (\E\|G\|^p)^{1/p}.
\end{align}
Moreover, by the strong-weak inequality for the gaussian measure (which follows directly from the gaussian concentration inequality, see, e.g., \cite[Lemma 3.1]{ledoux1991probability}),
\[ (\E\|G\|^p)^{1/p} \leq c_2 (\E \|G\| + \sqrt{p} {\cal R}(K^\circ) ).\]
When applied to $p=2c_0r\log \frac{em}{r}$, it follows from Chebyshev's inequality that for every $x\in \Sigma_r'$, with probability at least $1-\exp(-2c_0 r\log\frac{em}{r})$,
\[ \left\| \sum_{j=1}^m x_j Z_j \right\|
\leq c_3 L \left(\E \|G\| + {\cal R}(K^\circ) \sqrt{r \log(em/r)} \right),\]
and (1) is evident by the union bound.

The proof of (2) follows a similar path to (1), combined with tail-integration.

As for (3), by an identical argument to the one used previously, there are absolute constants $c_4$ and $c_5\geq 2$, such that
\begin{align*}
&\PP \left( \|Z\| \geq c_4 L \left(\E \|G\| + \sqrt{\log\left(em/r\right)} {\cal R}(K^\circ) \right) \right)  \\
&\leq \exp\left(-c_5\log\left(em/r\right)\right)
\leq \left( \frac{r}{em} \right)^2.
\end{align*}
Consider $m$ independent selectors (that is, $\{0,1\}$-valued random variables), defined by
\[
\delta_i=1 \ \ {\rm if} \ \ \|Z_i\| \geq c_4 L\left(\E \|G\| + \sqrt{\log\left(em/r\right)} {\cal R}(K^\circ) \right).
\]
Then $\PP(\delta_i=1) \leq (r/em)^2$, and by Bennett's inequality (see, e.g., \cite[Theorem 2.9]{boucheron2013concentration}), with probability at least $1-2\exp(-c_6r\log (em/r))$, we have that $|\{i : \delta_i=1\}| \leq r$, as claimed.
\end{proof}

\vspace{0.5em}
\begin{tcolorbox}
	Let $Z$ be the random variable as in Theorem \ref{thm:DM-main}.
	For $1\leq r\leq m$, denote by $\Omega_r(\mathbb{Z})$ the set of all $ (Z_i)_{i=1}^m$ such that both \eqref{eq:conc-event-2} and \eqref{eq:conc-event-norm-2}  hold.
\end{tcolorbox}

\subsection{Proof of Theorem \ref{thm:conc}---part 1}

Throughout this section, denote by $C_0,C_1$ the absolute constants from Lemma \ref{lem:sub-Gaussian.sparse}.
The constants $c,c_0,c_1,c_2$ etc.\ depend on $L$, $C_0$ and $C_1$.
For $1\leq r\leq m$ set
\begin{align} \label{eq:def.phi}
\phi(r) = C_1 L \left(\E\|G\| + {\cal R}(K^\circ) \sqrt{\log( em/r)}\right),
\end{align}
i.e.,  $\phi$ is the function appearing in \eqref{eq:conc-event-norm-2}.


\vspace{0.5em}
Rather than specifying the required values of $r,s,m,d$ and $H_{s,m}$, we shall collect conditions on those values, and show in Section \ref{sec:putting.concentration.proof.together} that the conditions  are satisfied under the assumptions of Theorem \ref{thm:conc}.

Finally, fix a realization $(X_i)_{i=1}^m$ for which $\sup_{u\in B_2^d}\frac{1}{m}\sum_{i=1}^m\inr{X_i,u}^2\leq 2$.

\vspace{0.5em}
\noindent \underline{\bf Step 1 --- $\clubsuit_u$:}
Let $s$  be specified in what follows, set $(Z_i)_{i=1}^m\in \Omega_s(\mathbb{Z})$ and recall that  $I_{u,s}$ is the set of indices corresponding to the $s$ largest coordinates of $(|\inr{X_i,u}|)_{i=1}^m$.
By the Cauchy-Schwartz inequality,
\begin{align*}
\|\clubsuit_u\|
= & \sup_{t \in K^\circ}  \frac{1}{\sqrt{m}} \sum_{i \in I_{u,s}} \inr{X_i,u} \inr{Z_i,t}
\\
\leq & \left(\frac{1}{m}\sum_{i \in I_{u,s}} \inr{X_i,u}^2 \right)^{1/2} \cdot \sup_{t \in K^\circ} \left( \sum_{i \in I_{u,s}} \inr{Z_i,t}^2 \right)^{1/2}.
\end{align*}
Since $|I_{u,s}|=s$, it follows from the choice of $s$ and the definitions of $H_{s,m}$ and $\Omega_s(\mathbb Z)$ that
\[
\|\clubsuit_u\| \leq H_{s,m}
\cdot C_0L\left(\E \|G\| + {\cal R}(K^\circ) \sqrt{s\log (em/s)} \right).
\]
Moreover, since
\[d^*(K) = \left( \frac{\E \|G\|}{{\cal R}(K^\circ)} \right)^2,\]
it is evident that for suitable constants $c_1$ and $c_1'$ the following holds: 
\begin{tcolorbox}
If
\begin{equation} \label{eq:conc-cond-1}
H_{s,m} \leq c_1 \eps
\qquad {\rm and} \qquad
s \log (em/s) \leq c_1' d^*(K),
\end{equation}
then for $(Z_i)_{i=1}^m\in \Omega_s(\mathbb{Z})$,
\[
\sup_{u \in B_2^d} \|\clubsuit_u\| \leq \eps \E \|G\|.
\]
\end{tcolorbox}

Using \eqref{eq:conc-mean-2} and an identical argument, the conditions in \eqref{eq:conc-cond-1} also imply that
\begin{equation} \label{eq:conc-sup-mean-large}
\sup_{u\in B_2^d} \E_Z \|\clubsuit_u\|
\leq \eps \E \|G\|.
\end{equation}

\vspace{0.5em}
\noindent \underline{\bf Step 2 --- $\diamondsuit_u$:}
Let $r$  be named in what follows and set $(Z_i)_{i=1}^m \in \Omega_r(\mathbb{Z})$.
Recall that
\[\xi_u= \frac{1}{\sqrt m}\max_{i \in I_{u,s}^c} |\inr{X_i,u}|\]
and thus $\xi_u\leq  H_{s,m}/ \sqrt{s}$.
Consider the function $\phi$  defined in \eqref{eq:def.phi} and set
\[J_{u,r} = \left\{ j\in\{1,\dots,m\} : \left\| \tfrac{1}{\sqrt{m}}\inr{X_j,u}  Z_j\right\| \geq \xi_u \phi(r) \right\}.\]
Clearly $|\frac{1}{\sqrt{m}}\inr{X_i,u}|\leq \xi_u$ for every $i\in I_{u,s}^c$; hence,
\begin{align}
\label{eq:coordinates.intersection}
I_{u,s}^c \cap J_{u,r}
\subset \left\{ j\in\{1,\dots,m\} : \|Z_j\| \geq \phi(r) \right\}.
\end{align}
Moreover, by the definition of $\Omega_r(\mathbb{Z})$,   $|\{ j: \|Z_j\| \geq \phi(r) \}|\leq r$.
Therefore,
\begin{align*}
\| \diamondsuit_u \|
& = \sup_{t \in K^{\circ}} \frac{1}{\sqrt{m}}\sum_{i \in I_{u,s}^c \cap J_{u,r}}  \inr{X_i,u} \inr{Z_i,t} \\
& \leq  \xi_u \cdot \sup_{t \in K^\circ} \max_{|J| = r} \sum_{j \in J} |\inr{Z_j,t}|
\\
& \leq \frac{H_{s,m}}{\sqrt s} \cdot \sqrt{r} \sup_{t \in K^\circ} \max_{|J| = r}  \left(\sum_{j \in J} \inr{Z_j,t}^2 \right)^{1/2} ,
\end{align*}
where the last inequality is evident by comparing the $\ell_1^r$ and $\ell_2^r$ norms.
Since $(Z_i)_{i=1}^m\in \Omega_r(\mathbb Z)$,
\[
\| \diamondsuit_u \|
\leq H_{s,m} \cdot \sqrt \frac{r}{s} C_0 L \left(\E\|G\| + {\cal R}(K^\circ) \sqrt{r \log\left( em/r\right)}\right),
\]
and in particular, for suitable constants $c_2$ and $c_2'$, we have the following:
\begin{tcolorbox}
If
\begin{equation} \label{eq:conc-cond-2}
H_{s,m} \sqrt{\frac{r}{s}}  \leq c_2  \eps \ \ {\rm and} \ \ r \log\left( em/r \right) \leq c_2'  d^*(K),
\end{equation}
then
\[
\sup_{u\in B_2^d} \| \diamondsuit_u \| \leq \eps \E \|G\|.
\]
\end{tcolorbox}

Next, let us estimate $\E_Z \| \diamondsuit_u \|$.

\begin{Lemma} 
If \eqref{eq:conc-cond-2} holds, then $\sup_{u\in B_2^d}\E_Z \|\diamondsuit_u\| \leq \eps \E\|G\|$.
\end{Lemma}
\begin{proof}
	Fix $u \in B_2^d$ and set $Q=| \{i\in I_{u,s}^c : \|Z_i\|\geq \phi(r) \}|$. For $q\geq 1$, let
	\[
	S(q)= \left( \sup_{t \in K^\circ} \max_{|J| \leq  q} \sum_{j \in J} \inr{Z_j,t}^2 \right)^{1/2},
\]
	with the convention that $\max \emptyset =0$.
	Therefore, by \eqref{eq:coordinates.intersection},
	\begin{align*}
	\|\diamondsuit_u\|
	&\leq \max_{i \in I_{u,s}^c} \left| \frac{1}{\sqrt m} \inr{X_i,u}\right| \cdot \sup_{t \in K^\circ} \sum_{j \in I_{u,s}^c\cap J_{u,r}} |\inr{Z_j,t}|  \\
	&\leq \frac{ H_{s,m} }{\sqrt{s} } \cdot \sqrt{Q } S(Q)  .
\end{align*}
	Since $H_{s,m}$ is independent of $(Z_i)_{i=1}^m$, all that is left is to show that $\E_Z \sqrt{Q } S(Q)\leq c\sqrt{r} \E \|G\|$ for a suitable constant $c$.
	And indeed,  observe that
	\begin{align*}
	\sqrt{Q} S(Q)
	&\leq  \IND_{[0,2r]}(Q)   \sqrt{ 2r } S(2r)  + \sum_{r<2^s<m} \IND_{(2^s,2^{s+1}]}(Q) \sqrt{ 2^{s+1} } S(2^{s+1})  \\
	&= (\ast) +(\ast\ast).
	\end{align*}
	By  \eqref{eq:conc-mean-2}, for  $1\leq q\leq m$,
	\[ \left( \E_Z S^2(q) \right)^{1/2}
	\leq C_0 L \left(\E\|G\| + {\cal R}(K^\circ) \sqrt{ q \log\left( em/q\right)}\right),  \]
	and  for  $r$ as in \eqref{eq:conc-cond-2} and the choice of $d^\ast(K)$,  $\E_Z (\ast)\leq c_1 \sqrt{r} \E \|G\|$.

	To estimate $\E_Z (\ast\ast)$, apply the Cauchy-Schwartz inequality to each term \linebreak $\E_Z \IND_{(2^s,2^{s+1}]}(Q)  S(2^{s+1})$:
	\[ \E_Z (\ast\ast)
	\leq \sum_{r<2^s< m}  \sqrt{ \PP( Q\geq 2^s) } \cdot  \sqrt{ 2^{s+1}}  C_0 L \left(\E\|G\| + {\cal R}(K^\circ) \sqrt{ 2^{s+1} \log\left( em/ 2^{s+1}\right)}\right) . \] 	
	Finally, just as in the proof of the third part of Lemma \ref{lem:sub-Gaussian.sparse}, $\PP( Q\geq 2^s)\leq 2\exp(-c_2 2^s)$ for $2^s\geq 2r$; hence,  by comparing the sum to a geometric progression,  $\E_Z (\ast\ast)\leq  c_3 \sqrt{r} \E \|G\|$, as claimed.
\end{proof}

\vspace{0.5em}
\noindent \underline{\bf Step 3 --- $\heartsuit_u$:}
By definition of $J_{u,r}$,
\begin{align*}
\heartsuit_u
&= \frac{1}{\sqrt{m}}\sum_{i \in I_{u,s}^c \cap J^c_{u,r}}  \inr{X_i,u} Z_i\\
&= \frac{1}{\sqrt{m}} \sum_{i \in I_{u,s}^c }\inr{X_i,u} \IND_{ \left\{ \left\| \frac{1}{\sqrt m} \inr{X_i,u} Z_i \right\| \leq \xi_u \phi(r) \right\} } Z_i.
\end{align*}
In particular, if we set
\begin{align*}
\mathcal{F}_u
&=\left\{ f_t(\cdot)=\inr{t,\cdot} \IND_{\{ \| \cdot \| \leq \xi_u \phi(r)\}} : t \in K^\circ \right\} \text{ and}\\
Y_i
&= \frac{1}{\sqrt{m}} \inr{X_i,u} Z_i \in \R^n,
\end{align*}
then
\begin{equation} \label{eq:representation}
\|\heartsuit_u\|
=\sup_{t \in K^\circ} \sum_{i \in I_{u,s}^c} f_t(Y_i).
\end{equation}

\begin{Remark}
Note that the sets $I_{u,s}$ depend on $(X_i)_{i=1}^m$ but not on $(Z_i)_{i=1}^m$.
\end{Remark}

Let us show that conditionally on $(X_i)_{i=1}^m$, each random variable $\|\heartsuit_u\|$ concentrates around its mean.
Equation  \eqref{eq:representation} means that $\|\heartsuit_u\|$ is a supremum of an empirical process and one may invoke the following version of Talagrand's concentration inequality for bounded empirical processes due to Klein and Rio \cite{klein2005concentration}.

\begin{Theorem} \label{thm:Klein-Rio}
Let $Y_1,\dots,Y_k$ be independent, set $\mathcal{F}$ to be a class of functions into $[-a,a]$ such that $\E f(Y_i) =0$ for every $i\in\{1,\dots,k\}$ and $f\in \mathcal{F}$.
Let $\sigma^2 =\sup_{f \in \mathcal{F}} \sum_{i=1}^k \E f^2(Y_i)$ and consider
\[
U=\sup_{f \in \mathcal{F}} \sum_{i=1}^k f(Y_i).
\]
Then for $x>0$,
\begin{align}
\label{eq:Klein-Rio}
\PP \left(|U-\E U| \leq x \right) \geq 1-2\exp\left(-\frac{-x^2}{2(\sigma^2+2a \E U) + 3 a x} \right).
\end{align}
\end{Theorem}

In the case that interests us, $k=|I_{u,s}^c|=m-s$ and $\mathcal{F}=\mathcal{F}_u$.
We start by verifying the  assumptions of Theorem \ref{thm:Klein-Rio} for the empirical process from \eqref{eq:representation}.

\begin{Lemma}
\label{lem:talagrand.constants}
	The set $\mathcal{F}_u$ satisfies the assumptions in Theorem \ref{thm:Klein-Rio} with $a= \frac{ H_{s,m}}{ \sqrt s} \phi(r)$.
	Moreover,  $\sigma^2\leq 4\mathcal{R}^2(K^\circ)$  and $\E U \leq cL \E\|G\|$ for an absolute constant $c$.
\end{Lemma}
\begin{proof}
	First note that $\E f(Y_i)=0$ for every $f=f_t\in \mathcal{F}_u$ and $1\leq i \leq k$.
	Indeed, since $\|t\|=\|-t\|$ and $Y_i$ is symmetric,
	\[ \E f_t(Y_i)
	=\E_Z \inr{t,Y_i}\IND_{\{\|Y_i\| \leq \xi_u \phi(r)\}}
	=0.\]
	Turning to the estimate on  $a$, let $f=f_t$ for $t\in K^\circ$.
	Then for every $y \in \R^n$,
	\begin{align*}
	|f_t(y)|
	&= |\inr{t,y}| \ \IND_{\{ \| y \| \leq \xi_u \phi(r)\}} \\
	&\leq  \|y\|\IND_{\{ \|y\| \leq \xi_u \phi(r)\}}
	\leq \xi_u\phi(r)
	\leq \phi(r) \frac{H_{s,m}}{\sqrt s}
	\end{align*}
	by the (uniform) upper estimate on $\xi_u$.
	
As for $\sigma^2$, recall that $\frac{1}{m}\sum_{i=1}^m \inr{X_i,u}^2 \leq 2$ and that  $Z$ is isotropic.
	Hence
	\begin{align*}
	\sigma^2
	&\leq \sup_{t \in K^\circ} \sum_{i\in I_{u,s}^c} \frac{1}{m}\inr{X_i,u}^2 \E \inr{Z_i,t}^2 \\
	&\leq 2\sup_{t \in K^\circ} \|t\|_2^2 = 2{\cal R}^2(K^\circ).
	\end{align*}

	Finally, since $(Z_i)_{i=1}^m$ and $(\eps_i Z_i)_{i=1}^m$ have the same distribution by the symmetry of $Z$, a standard Bernoulli contraction argument (see, e.g., \cite[Theorem 4.4]{ledoux1991probability}) applied conditionally on $(X_i)_{i=1}^m$ and $(Z_i)_{i=1}^m$ shows that
	\begin{align*}
	\E U
	& = \E_Z  \E_\eps \sup_{t \in K^\circ} \frac{1}{\sqrt{m}} \sum_{i\in I_{u,s}^c} \inr{X_i,u}\IND_{\{\|Y_i\| \leq \xi_u \phi(r)\}} \varepsilon_i\inr{Z_i,t}
\\
	& \leq \E_Z \left\| \frac{1}{\sqrt{m}} \sum_{i\in I_{u,s}^c}\inr{X_i,u} Z_i  \right\|
	=(\ast) .
	\end{align*}
	Since $(Z_i)_{i=1}^m$ are independent and $L$-subgaussian,  gaussian domination and the assumption that $\frac{1}{m}\sum_{i=1}^m \inr{X_i,u}^2\leq 2$ imply that $(\ast)\leq cL \E \|G\|$	
	for an absolute constant $c$.
\end{proof}

The next step is to control  the denominator in the probability estimate in  \eqref{eq:Klein-Rio}.
Consider $0<x<\E \|G\|$, and note that by Lemma \ref{lem:talagrand.constants},
\begin{equation*}
2(\sigma^2+2a \E U) + 3 a x
\leq c_3 \left({\cal R}^2(K^\circ) + \frac{H_{s,m}}{\sqrt s} \phi(r) \E \|G\|\right) .
\end{equation*}
It follows from the definition of $\phi$ that if 
\[
m \leq c_4 r\exp(d^*(K))
\]
then $\phi(r)\leq c_5 \E\|G\|$, and therefore
\begin{equation*}
2(\sigma^2+2a \E U) + 3 a x
\leq c_6 \left({\cal R}^2(K^\circ) + \frac{H_{s,m}}{\sqrt s} (\E \|G\|)^2\right) .
\end{equation*}

Setting $x=\eps \E\|G\|$ and since $U=\|\heartsuit_u\|$ (see \eqref{eq:representation}), it follows from Theorem \ref{thm:Klein-Rio} that for suitable constants $c_7$, $c_7^\prime$ and $c_7^{\prime \prime}$ we have that:
\begin{tcolorbox}
If
\begin{equation}
\label{eq:conc-cond-3}
\frac{H_{s,m}}{\sqrt{s}} \leq \frac{ c_7 }{ d^\ast(K)}\quad\text{and}\quad
m \leq c_7^\prime r\exp(d^*(K)),
\end{equation}
then for every $u\in B_2^d$, with $\PP_Z$-probability at least $1-2\exp( - c_7^{\prime \prime} \eps^2 d^*(K))$,
\[
\left |\|\heartsuit_u\| -\E_Z \|\heartsuit_u\| \right|
\leq \eps \E \|G\|.
\]
\end{tcolorbox}

\begin{Remark}
It is worth stressing that in contrast to Step 1 and Step 2, the probability estimate here does not hold uniformly for $u\in B_2^d$, but only for a single (arbitrary) $u$.
\end{Remark}

\subsection{Proof of Theorem \ref{thm:conc}}
\label{sec:putting.concentration.proof.together}

First, let us show that for well-chosen $s,r,m$, and $d$,  all the requirements  collected in \eqref{eq:conc-cond-1}, \eqref{eq:conc-cond-2}, and \eqref{eq:conc-cond-3} are indeed satisfied.
Let $c_1,\dots,c_4$ be suitable constants that turn out to depend only on $L$, $C_0$, and $C_1$.
Set
\[s=c_1\frac{d^*(K)}{\log\left( em/d^*(K)\right)}
\ \ \ {\rm and} \ \ \
r = c_2 \min\{ \eps^2 d^*(K) , s\},\]
and let $m\in[c_3 d^*(K), c_4\exp(d^\ast(K))]$.
Then \eqref{eq:conc-cond-1} and  \eqref{eq:conc-cond-2} are satisfied for suitable choices of $c_1,c_2$.
Moreover, as noted previously, one may assume without loss of generality that $\varepsilon^2 d^\ast(K)\geq  1$---otherwise the claim in Theorem \ref{thm:conc} is trivially true; thus, $m$ satisfies  \eqref{eq:conc-cond-3} for suitable choice of  $c_4$.

\vspace{0.5em}
Next, collecting the requirements on $H_{s,m}$ from \eqref{eq:conc-cond-1}, \eqref{eq:conc-cond-2}, and \eqref{eq:conc-cond-3}, one has to verify that
\begin{align}
\label{eq:req.Wsm}
H_{s,m}
\leq c_5 \min\left\{\eps,\eps \sqrt{\frac{s}{r}},  \frac{\sqrt s}{d^\ast(K)} \right\}.
\end{align}
To that end, recall that $r\leq c_2 s$ and that $\eps\geq 1/\sqrt{d^\ast(K)}$; in particular, the third term in \eqref{eq:req.Wsm} is the  dominant one, and \eqref{eq:req.Wsm} is equivalent to having $H_{s,m}\leq c_6 \sqrt{s}/d^\ast(K)$, which was assumed to hold.

Therefore, if
\[(Z_i)_{i=1}^m \in \Omega_s(\mathbb{Z}) \cap \Omega_r(\mathbb{Z}),\]
then
\[\max\left\{ \|\clubsuit_u\|,  \|\diamondsuit_u\|, \E_Z\|\clubsuit_u\|, \E_Z \|\diamondsuit_u\| \right\}
\leq \eps \E\|G\|.\]
To estimate the probability of the event $\Omega_s(\mathbb{Z}) \cap \Omega_r(\mathbb{Z})$, note that $s\log(em/s)\ge c_7 d^\ast(K)$ by the choice  of $s$ and since $m\geq c_3d^\ast(K)$.
In particular, as $\varepsilon\leq 1$, we have that $s\log(em/s)\geq c_7\eps^2 d^\ast(K)$.
Repeating  the  argument for $r$, it is evident from Lemma \ref{lem:sub-Gaussian.sparse} that
\begin{align*}
\PP_Z \left( \Omega_r(\mathbb{Z})\cap \Omega_s(\mathbb{Z}) \right)
&\geq 1-2\exp\left(-C_2 s\log\left(\frac{em}{s}\right) \right) - 2\exp\left(-C_2 r\log\left(\frac{em}{r}\right) \right) \\
&\geq 1-2\exp(-C_3 \varepsilon^2 d^\ast(K))
\end{align*}
for  constants $C_2$ and $C_3$ that depend on $c_7$ and on the absolute constants $C_0^\prime$ and $C_1^\prime$ from Lemma \ref{lem:sub-Gaussian.sparse}.

\vspace{1em}
Finally, one has to show that $\|\heartsuit_u\|$ concentrates around its mean, and that the required concentration holds uniformly in $u\in W$.
As was noted previously, for any $u \in W$, with $\PP_Z$-probability at least $1-2\exp(-c_8 \eps^2 d^*(K))$
\[
\left| \ \|\heartsuit_u\|-\E_Z \|\heartsuit_u\| \ \right| \leq \eps \E \|G\|,
\]
and  $\log |W| \leq d \log(5/\eps)$.
If
\[ d \leq  c_9 \frac{\eps^2  d^*(K)}{\log(5/\eps)}\]
for a suitable constant $c_9$, then by the union bound, with $\PP_Z$-probability at least \linebreak $1-2\exp(-c_{10}\eps^2 d^*(K))$,
\begin{equation} \label{eq:putting-together-uniform}
\max_{u \in W} \left| \ \|\heartsuit_u\|-\E_Z \|\heartsuit_u\| \ \right| \leq \eps \E \|G\|.
\end{equation}

Taking the intersection of $\Omega_r(\mathbb{Z})\cap \Omega_s(\mathbb{Z})$ and the high-probability event on which \eqref{eq:putting-together-uniform} holds completes the proof.
\qed

\section{A lower bound on the conditional expectation---proof of \eqref{eq:conc-2}}
\label{sec:exp.lower}

Let us turn to verifying  \eqref{eq:conc-2} from Lemma \ref{lemma:reduction}.
Once again,  recall that $Z$ is a symmetric, isotropic, $L$-subgaussian random vector in $\mathbb{R}^n$ that has iid coordinates.

\begin{Theorem}
\label{thm:lower.bound.expectation}
	Let $\delta,\eta>0$.
	There are constants $c_1$ and $c_2$ that depend on $\delta$ and $\eta$ such that the following holds.
	Let $m\geq c_1\log n$ and fix $v\in S^{d-1}$.
	If
	\begin{align}
	\label{eq:SB.X}
	|\{i \in\{1,\dots,m\} :  |\inr{X_i,v}|\geq \eta \} | \geq \delta m ,
	\end{align}
	then
	\[ \E_Z\Psi(v) \geq c_2 \E\|G\|.\]
\end{Theorem}

\begin{Remark}
A version of Theorem \ref{thm:lower.bound.expectation} was  established in \cite{mendelson2022isomorphic} under different conditions (most notably that $m \geq c n$ instead of $m\geq c \log n$ and without the assumption that the coordinates of $Z$ are iid).
The idea was to find a lower bound on the Bernoulli processes
\[ \E_\eps \sup_{t\in K^\circ} \frac{1}{\sqrt{m}}\sum_{i=1}^m \varepsilon_i \inr{X_i,v} \inr{Z_i,t} ,\]
conditioned on a typical realization of $(X_i)_{i=1}^m$ and $(Z_i)_{i=1}^m$, and the proof relied on the construction of  an approximation of $K^\circ \subset \R^n$ consisting of $\exp(c m)$ points.
Unfortunately, in the worst case it forced $m$  to be proportional to $n$.
Since the aim here is for $m$ to be as small as possible---specifically a low-degree polynomial function of the critical dimension $d^\ast(K)$---a different argument is called for.
\end{Remark}

Note that the condition $m\geq c_1 \log n$ in Theorem \ref{thm:lower.bound.expectation} is satisfied in the setting of Theorem \ref{thm:DM-main}.
	Indeed, $m$ is proportional to $(d^\ast(K))^\alpha$  for $\alpha\geq 4$ (see \eqref{eq:cond-on-m}),  $d^\ast(K)\geq C_0\log n$ for an absolute constant $C_0$, and  $n\geq C(L,\alpha)$ for a constant $C$ that we are free to choose as large as needed.

	We will verify in  Lemma \ref{lem:X.SB}  that there are constants $\delta,\eta>0$ depending only on $L$ such that \eqref{eq:SB.X} is satisfied.
	
	Finally, by gaussian domination, if  $\frac{1}{m}\sum_{i=1}^m \inr{X_i,v}^2\leq 2$ then  $\E_Z\Psi(v) \leq C_2 L \E \|G\|$ for an absolute constant $C_2$.
	Thus, in the setting of Theorem \ref{thm:DM-main},
	\[\E_Z \Psi(v) \in [C_3 \E\|G\|, C_4 \E\|G\|]\]
	for constants $C_3$ and $C_4$ that depend only on $L$.

\vspace{0.5em}
Recall that for $1\leq i\leq m$, $\lambda_i^v=m\int_{(\frac{i-1}{m}, \frac{i}{m}]}  F_{ \inr{X,v} }^{-1}(p) \,dp$.

\begin{Lemma}
\label{lem:X.SB}
Let $X$ satisfy $L_4-L_2$ norm equivalence with constant $L$.
	There are constants $\delta,\eta>0$ that depend only on $L$ such that the following holds.
	If $m\geq 2/\delta$ and 
	\begin{align}
	\label{eq:small.ball.X}
	 \left( \frac{1}{m}\sum_{i=1}^m \left| \inr{X_i,v}^\sharp -  \lambda_i^v\right|^2 \right)^{1/2}
\leq \eta \sqrt{\delta},
	\end{align}
	then
	\[|\{i \in\{1,\dots,m\} :  |\inr{X_i,v}|\geq \eta \} | \geq \delta m. \] 
\end{Lemma}

Note that \eqref{eq:small.ball.X} holds for  $(X_i)_{i=1}^m\in\Omega(\mathbb{X})$. 
Indeed, by \eqref{eq:SW-2-est-DM}, the left hand side in  \eqref{eq:small.ball.X}  is smaller than $2\varepsilon$, and we may assume that  $\eps\leq C$ for an arbitrarily small constant $C=C(L)$, e.g., for $C=\eta\sqrt\delta$.

\begin{proof}
	By the $L_4-L_2$ norm equivalence and the Paley-Zygmund inequality (see, e.g., \cite[Corollary 3.3.2]{de2012decoupling}), there are constants $\eta$ and $\delta$ that depend only on $L$ such that
	\[\mu^v\left( \left( -2\eta,2\eta\right)^c \right)
	= \PP\left( |\inr{X,v} |\geq 2\eta \right)
	 \geq 4\delta.\]
	In particular $\PP( \inr{X,v} \geq 2\eta )\geq 2\delta$ or $\PP( \inr{X,v} \leq - 2\eta )\geq 2\delta$, and we focus without loss of generality on the latter case.
	Then clearly  
	\[\lambda^v_i\leq -2\eta \quad\text{for } i=1,\dots, 2\delta m.\]
	It is straightforward to verify that if  $\inr{X_{\delta m},v}^\sharp > -\eta$, then \eqref{eq:small.ball.X} cannot be true.
	In particular, \eqref{eq:small.ball.X} implies that $|\{i :  \inr{X_i,v} \leq -\eta \} | \geq \delta m$ and  \eqref{eq:SB.X} holds.
\end{proof}

\subsection{Proof of Theorem \ref{thm:lower.bound.expectation}}
The proof of Theorem \ref{thm:lower.bound.expectation} is based on two facts.
To formulate them, let $(x^\ast_i)_{i=1}^m$ be the non-increasing rearrangement of $(|x_i|)_{i=1}^m$.

\begin{Lemma}
\label{lem:permutation.sorting.shuffeling}
	Let $w$ be a symmetric random variable, set $(w_i)_{i=1}^n$ to be independent copies of $w$ and put $W=(w_i)_{i=1}^n$.
	Let $\eps=(\eps_i)_{i=1}^n$ be the Bernoulli vector and set $\pi$ to be the uniform distribution on the permutations of $\{1,\dots,n\}$.
	If $W$, $\eps$ and $\pi$ are independent, then
	\[ \E_W \sup_{t\in K^\circ} \sum_{i=1}^n w_i t_i
	=\E_{W,\pi,\eps} \sup_{t\in K^\circ} \sum_{i=1}^n  w^\ast_{i} \eps_{\pi(i)}t_{\pi(i)}.	 \]
\end{Lemma}
\begin{proof}
	The key is the (rather obvious) observation that $	(w_1,\dots,w_n) $ has the same distribution as $(\eps_1 w^\ast_{\pi(1)},\dots, \eps_n w^\ast_{\pi(n)} )$.
	Therefore,
	\begin{align*}
	\E_W \sup_{t\in K^\circ} \sum_{i=1}^n w_i t_i
	&=\E_{W,\pi,\eps} \sup_{t\in K^\circ} \sum_{i=1}^n \eps_i w^\ast_{\pi(i)} t_i \\
	&=\E_{W,\pi,\eps} \sup_{t\in K^\circ} \sum_{i=1}^n  w^\ast_{i} \eps_{\pi(i)}t_{\pi(i)},	
	\end{align*}
	where the last equality holds because $\pi^{-1}$ and $\pi$ have the same distribution.
\end{proof}

\begin{Lemma}
\label{lem:oder.stat}
	Let $\delta>0$.
	There are constants $c_1,\dots,c_4$ that depend only on $\delta$ and $L$ such that the following holds.
	Let $m \geq c_1 \log n$, consider $I\subset\{1,\dots,m\}$, $|I|\geq\delta m$ and set for $j=1,\dots,n$,
	\begin{align}
	\label{eq:def.Y.j}
	Y_j=\frac{1}{\sqrt m} \sum_{i\in I}  \inr{Z_i, e_j}.
	\end{align}
	Then for any $1\leq k \leq n$, $\E|Y_k|\geq c_2$.
	Also, for any $1\leq k\leq c_3n$,
	\begin{align}
	\label{eq:def.Y.rearrangment}
	 \E Y^\ast_k \geq c_4 \sqrt{ \log \left( \frac{n}{k} \right) } .
	 \end{align}
\end{Lemma}

The proof of Lemma \ref{lem:oder.stat} relies on the an estimate due to Montgomery-Smith \cite{montgomery1990distribution}.
To formulate it, denote by $\lfloor u \rfloor$  the largest integer that is smaller than $u$.

\begin{Lemma} \label{lem:MS}
	There are absolute constants $c_1,c_2$ and $c_3$ such that the following holds.
	Let $\eps=(\eps_i)_{i=1}^m$ be the Bernoulli vector and set $0< u\leq m-2$.
	Then, for every  $x\in\R^m$, with probability at least $c_1\exp(-c_2u)$,
	\begin{align}
	\label{eq:M.S}
	\sum_{i=1}^m  x_i \eps_i \geq
	c_3 \left( \sum_{i=1}^{\lfloor  u\rfloor}  x^\ast_i + \sqrt{u}  \left( \sum_{i= \lfloor u\rfloor + 1 }^m  (x^\ast_i)^2 \right)^{1/2}\right).
	\end{align}
\end{Lemma}

\begin{proof}[Proof of Lemma \ref{lem:oder.stat}]
	We only present a proof of \eqref{eq:def.Y.rearrangment}.
	The proof that $\min_{1\leq k\leq n}\E|Y_k|\geq c$ follows  a similar (but simpler) path.

	Since $Z$ has iid coordinates, the random variables $(Y_j)_{j=1}^n$ are independent and each $Y_j$ has the same distribution as $Y=Y_1$.
	It is straightforward to verify that if $u$ satisfies that $\PP ( |Y|\geq  u ) \geq c_1 k/n$, then $\PP( Y^\ast_k \geq u) \geq 1/2$ as long as $c_1$ is a well-chosen absolute constant; in particular,  $\E Y^\ast_k \geq u/2$.
	The rest of the proof is devoted to establishing that $u$ of the order $\sqrt{\log \frac{n}{k}}$ is a valid choice.
	
	Since $Z$ is $L$-subgaussian and isotropic, by the Paley–Zygmund inequality,  there are constants $c_2$ and $c_3$ depending only on $L$ such that $\PP(|\inr{Z,e_1}|\geq c_2)\geq 2c_3$.
		Let $\varepsilon=(\varepsilon_i)_{i=1}^m$ be the Bernoulli vector, independent of $(Z_i)_{i=1}^m$.
	Setting
	\[ I'=\{ i \in I :  |\inr{Z_i,e_1}|\geq c_2 \}, \]
	then clearly  $\PP\left( |I'|\geq c_3 |I|\right)\geq 1/2$.
	By the contraction inequality for Bernoulli processes 	and the symmetry of $Z$,
	\begin{align*}
	\PP ( |Y| \geq u )
	&= \E_Z \PP_\eps\left( \left| \frac{1}{\sqrt m} \sum_{i\in I}  \eps_i \inr{Z_i, e_1} \right| \geq u \right) \\
	&\geq \E_Z \IND_{\{ |I'|\geq c_3 |I| \} } \frac{1}{2}\PP_\eps\left( \left| \frac{1}{\sqrt m} \sum_{i\in I'}  c_2 \eps_i \right| \geq u \right).
	\end{align*}
	Fix a realization  $(Z_i)_{i=1}^n$ satisfying  $|I'|\geq c_3|I|$.
	It suffices to show that
	\begin{align}
	\label{eq:some.equation.gaussian.approx}
	\PP_\eps\left( \left| \frac{1}{\sqrt m} \sum_{i\in I'}  c_2 \eps_i \right| \geq c_4 \sqrt{ \log\left(\frac{n}{k} \right) } \right) \geq \frac{ 4c_1 k}{n}
	\end{align}
	for a suitable constant $c_4$. 
	To that end, by Lemma \ref{lem:MS} there are absolute constants $c_5,c_6,c_7$ such that for $0\leq u^2 \leq |I'|-2$, with probability at least $c_5\exp(-c_6 u^2)$,
	\begin{align*}
  \left| \sum_{i\in I'} \frac{c_2}{\sqrt m}  \eps_i \right|
	 \geq c_7 \left( \sum_{i=1}^{ \lfloor u^2 \rfloor} \frac{c_2}{\sqrt m} + u \left(\sum_{i=\lfloor u^2 \rfloor+1}^{|I'| } \frac{c_2^2}{ m} \right)^{1/2} \right)
	 =A(u).
	 \end{align*}
	 Let $u=c_8\sqrt{\log \frac{n}{k}}$ for a (small) constant $c_8=c_8(\delta)$ and observe that for that  choice of $u$, $A(u)\geq c_9(\delta) u$.
	 Indeed, as was assumed, 
	 \[|I'|\geq c_3 |I| \geq c_3\delta m\geq  c_3\delta \log n;\]
	 in particular  $|I'|-(\lfloor u^2 \rfloor +1 )\geq \frac{c_3 \delta m}{2}$ and thus $A(u)\geq c_7 u (\sum_{i=\lfloor u^2 \rfloor+1}^{|I'| } \frac{c_2^2}{ m} )^{1/2} \geq c_9(\delta)u$.
	 Therefore
	 \begin{align}
	 \label{eq:mont.bound}
	 \PP_\eps\left(\left| \frac{1}{\sqrt m} \sum_{i\in I'}  c_2 \eps_i \right| \geq c_{9} \cdot u \right)
	 &\geq c_5 \exp\left( -c_6 c_8^2  \log \left( \frac{n}{k} \right) \right),
	 \end{align}
	 and  if $c_8$ is sufficiently small and $k\leq c_{10} n$, then the right hand side in \eqref{eq:mont.bound} is at least $c_1\frac{k}{n}$.
	 This completes the proof.
\end{proof}

\begin{proof}[{\bf Proof of Theorem \ref{thm:lower.bound.expectation}}]	
	Set
	\[I=\{i\in\{1,\dots,m\} : |\inr{X_i,v}|\geq \eta\}\]
	and thus $|I|\geq \delta m$.
	Let $\varepsilon=(\varepsilon_i)_{i=1}^m$ be the Bernoulli vector, independent of $(Z_i)_{i=1}^m$.
	By the symmetry of $Z$ and the contraction principle for Bernoulli processes,
	\begin{align}
	\label{eq:compare.Z.G.1}
	\begin{split}
	\E_Z \sup_{t\in K^\circ} \sum_{i=1}^m \frac{1}{\sqrt m}\inr{X_i,v} \inr{Z_i,t}
	&=\E_{Z} \E_\eps \sup_{t\in K^\circ} \sum_{i=1}^m \eps_i \frac{1}{\sqrt m}\inr{X_i,v} \inr{Z_i,t} \\
	&\geq  \E_{Z} \E_\eps \sup_{t\in K^\circ} \sum_{i\in I} \eps_i \frac{1 }{\sqrt m}\eta  \inr{Z_i,t}.
	\end{split}
	\end{align}
	Also, for every $t\in K^\circ$, using the symmetry of $Z$ once again,
	\begin{align*}
	\sum_{i\in I} \eps_i  \frac{1}{\sqrt m} \inr{Z_i,t}
	&=\sum_{i\in I} \sum_{j=1}^n \frac{1}{\sqrt m}  \eps_i \inr{Z_i,e_j} t_j \\
	&=\sum_{j=1}^n \left(  \sum_{i\in I}  \frac{1}{\sqrt m}  \eps_i \inr{Z_i,e_j} \right) t_j
	= \sum_{j=1}^n  Y_j t_j.
	\end{align*}
	In particular, it follows from  \eqref{eq:compare.Z.G.1}  that
	\begin{align}
	\label{eq:compare.Z.G.2}
	\E_Z \sup_{t\in K^\circ} \sum_{i=1}^m \inr{X_i,v} \inr{Z_i,t}
	\geq \eta \E_{Z} \sup_{t\in K^\circ} \sum_{j=1}^n   Y_j t_j.
	\end{align}
	To complete the proof one has to ``replace'' $(Y_j)_{j=1}^m$ by independent gaussian random variables $(g_j)_{j=1}^m$. To that end, observe that by Lemma \ref{lem:oder.stat} there is a constant $c_1=c_1(L,\delta)$ such that for every $1\leq k\leq n$, $\E|Y_k|\geq c_1$.
	And, by Lemma \ref{lem:permutation.sorting.shuffeling},
	\begin{align*}
	&\E_{G} \sup_{t\in K^\circ} \sum_{j=1}^m g_j t_j
	=\E_{G,\pi,\eps} \sup_{t\in K^\circ} \sum_{j=1}^n g_j^\ast \eps_{\pi(j)} t_{\pi(j)} \\
	&=\E_{G}\E_{\eps,\pi} \sup_{t\in K^\circ} \sum_{j=1}^n \left( \frac{ g_j^\ast }{ \E_Y Y^\ast_j + c_1 }  \right) ( \E_Y Y^\ast_j + c_1)\eps_{\pi(j)} t_{\pi(j)} \\
	&\leq \left( \E_G \max_{1\leq j\leq n} \frac{ g_j^\ast }{ \E_Y Y^\ast_j + c_1 } \right) \cdot \left(
	 \E_{\eps ,\pi}\sup_{t\in K^\circ} \sum_{j=1}^n  ( \E_Y Y^\ast_j + c_1)\eps_{\pi(j)} t_{\pi(j)} \right)
	 =(\ast) \cdot (\ast\ast),
	\end{align*}
	where the last inequality follows from the contraction principle for Bernoulli processes (applied conditionally with respect to $G$).
	
	All that remains is to  show that $(\ast)\leq c(c_1,L,\delta)$ and that $(\ast\ast)\leq 2\E_Y \sup_{t\in K^\circ} \sum_{j=1}^n Y_jt_j$.
	
	\vspace{0.5em}
	To estimate the first term in $(\ast\ast)$, note that by Jensen's inequality followed by Lemma \ref{lem:permutation.sorting.shuffeling},
	\begin{align*}
	\E_{\eps,\pi} \sup_{t\in K^\circ} \sum_{j=1}^n  ( \E_Y Y^\ast_j ) \eps_{\pi(j)} t_{\pi(j)}
	&\leq \E_{\eps,\pi,Y} \sup_{t\in K^\circ} \sum_{j=1}^n Y^\ast_j\eps_{\pi(j)} t_{\pi(j)} \\
	&= \E_{Y} \sup_{t\in K^\circ} \sum_{j=1}^n  Y_j t_j.
	\end{align*}
	As for the second term in $(\ast\ast)$,  invoking Lemma \ref{lem:oder.stat},  $\E_Y |Y_j|\geq c_1$ for $1\leq j\leq n$.
	Hence, by  contraction
	\begin{align*}
\E_{\eps,\pi} \sup_{t\in K^\circ} \sum_{j=1}^n  c_1\eps_{\pi(j)} t_{\pi(j)}
	&\leq \E_{\eps,\pi}  \sup_{t\in K^\circ} \sum_{j=1}^n \E_Y |Y_{\pi(j)}| \eps_{\pi(j)} t_{\pi(j)} \\
	&\leq \E_Y \sup_{t\in K^\circ} \sum_{j=1}^n Y_j t_j,
	\end{align*}
	as required.
	
	\vspace{0.5em}
To control $(\ast)$, a standard binomial estimate shows that there are absolute constants $c_2$ and $c_3$ such that for $u\geq c_3$,
	 \begin{align}
	 \label{eq:gauss.order.stat}
	 \PP\left( g^\ast_j\geq  u \sqrt{ \log (en/j)} \right)
	 	\leq 2\exp\left(-c_4 u^2 j\log (en/j) \right).
	 \end{align}
	By Lemma \ref{lem:oder.stat} there is $c_5=c_5(L,\delta)$ such that  $\E_Y Y^\ast_j + c_1 \geq c_5\sqrt{\log (en/j)}$ for every $1\leq j\leq n$.
	Thus, \eqref{eq:gauss.order.stat} and the union bound imply that
	 \begin{align*}
	 \PP_G\left( \max_{1\leq j\leq n} \frac{ g_j^\ast }{ \E_Y Y^\ast_j + c_1 } \geq \frac{u}{c_5} \right)
	 &\leq \sum_{j=1}^n 2\exp \left(-c_4 u^2 j\log \left(\frac{en}{j}\right)  \right)\\
	 &\leq 2\exp\left(-c_6 u^2\right),
	 \end{align*}
	and the claim follows from tail-integration.
\end{proof}

\vspace{1em}
\noindent
{\bf Acknowledgements:}
The first author is grateful for financial support through the Austrian Science Fund (FWF) projects ESP-31N and P34743N.

\bibliographystyle{abbrv}

\end{document}